\documentclass[12pt]{amsart}

\usepackage{mathrsfs,amssymb}
\usepackage{mathtools} \mathtoolsset{showonlyrefs,showmanualtags}
\usepackage[
    colorlinks=true,       
    linkcolor=magenta,          
    citecolor=magenta,        
    filecolor=magenta,      
    urlcolor=magenta,           
    ]{hyperref}
\usepackage{stmaryrd,mathrsfs}
\usepackage{xcolor}
\usepackage{amsmath,amsthm,amscd,amssymb,amsfonts, amsbsy}
\usepackage{latexsym}
\usepackage{exscale}

\newtheorem*{theorem*}{Theorem}
\newtheorem{theorem}{Theorem}[section]
\newtheorem{lemma}[theorem]{Lemma}

\newtheorem{cor}[theorem]{Corollary}

\theoremstyle{definition}

\theoremstyle{remark}
\newtheorem{remark}[theorem]{Remark}

\numberwithin{equation}{section}

\DeclareMathOperator*{\esssup}{ess\,sup}

\DeclareMathOperator*{\BMO}{BMO}

\begin{document}

\author{Israel P. Rivera-R\'{\i}os}
\address[1]{Department of Mathematics, University of the Basque Country, Bilbao, Spain}
\address[2]{BCAM - Basque Center for Applied Mathematics}
\email{petnapet@gmail.com}

\title[Improved $A_1-A_\infty$ and related estimates]{ Improved $A_1-A_\infty$ and related estimates for commutators of rough singular integrals}

\thanks{This research was supported by the Basque Government through the BERC 2014-2017 program and by the Spanish Ministry of Economy and Competitiveness MINECO through BCAM Severo Ochoa excellence accreditation SEV-2013-0323 and  also through the project MTM2014-53850-P}

\begin{abstract}An $A_1-A_\infty$ estimate improving a previous result in \cite{PRRR} for $[b,T_\Omega]$ with $\Omega\in L^\infty(\mathbb{S}^{n-1})$ and $b\in \BMO$ is obtained. Also a new result in terms of the ${A_\infty}$ constant and the one supremum $A_q-A_\infty^{\exp}$ constant is proved, providing a counterpart for commutators of the result obained in \cite{Li}. Both of the preceding results rely upon a sparse domination result in terms of bilinear forms which is established using techniques from \cite{LRough}.\end{abstract}


\maketitle

\section{Introduction}
We recall that a weight $w$, namely a non negative locally integrable function,  belongs to $A_p$ if
\[[w]_{A_p}=\sup_Q\left(\frac{1}{|Q|}\int_Qw\right)\left(\frac{1}{|Q|}\int_Qw^{\frac{1}{1-p}}\right)^{p-1}<\infty\quad 1<p<\infty\]
or in the case $p=1$ if
\[[w]_{A_1}=\esssup_{x\in\mathbb{R}^n}\frac{Mw(x)}{w(x)}<\infty.\]

Given $\Omega\in L(\mathbb{S}^{n-1})$ with $\int_{\mathbb{S}^{n-1}}\Omega=0$ we define the rough singular integral $T_\Omega$ by
\[T_\Omega f(x) = pv \int_{\mathbb{R}^n}\frac{\Omega(y')}{|y|^n}f(x-y)dy\]where $y'=\frac{y}{|y|}$. 

During the last  years an increasing interest on the study of the sharp dependence on the $A_p$ constants of rough singular integrals has appeared. In particular it was established in \cite{HRT} that 
\[
\|T_\Omega\|_{L^2(w)}\leq c_n\|\Omega\|_{L^\infty(\mathbb{S}^{n-1})} [w]_{A_2}^2.
\]
Recently the following sparse domination (very recently reproved in \cite{LRough} for the case $\Omega\in L^\infty(\mathbb{S}^{n-1})$) was established in \cite{CACDiPO}.
\stepcounter{equation}
\begin{theorem*}
For all $1<p<\infty$, $f\in L^{p}(\mathbb{R}^{n})$ and $g\in L^{p'}(\mathbb{R}^{n})$,
we have that
\begin{equation}
\tag{\theequation}
\label{eq:SpTO}
\Big|\int_{\mathbb{R}^{n}}T_\Omega(f)gdx\Big|\leq c_n C_Ts'\sup_{\mathcal{S}}\sum_{Q\in\mathcal{S}}\Big(\int_{Q}|f|\Big)\Big(\frac{1}{|Q|}\int_{Q}|g|^{s}\Big)^{1/s},
\end{equation}
where each $\mathcal{S}$ is a sparse family of a dyadic lattice $\mathcal{D}$,
\[
\begin{cases}
1<s<\infty & \text{if } \Omega\in L^{\infty}(\mathbb{S}^{n-1})\\
q' \le s<\infty& \text{if } \Omega\in L^{q,1}\log L(\mathbb{S}^{n-1})
\end{cases}
\] and
\[
C_{T}=\begin{cases}
\|\Omega\|_{L^{\infty}(\mathbb{S}^{n-1}),} & \text{if }\Omega\in L^{\infty}(\mathbb{S}^{n-1})\\
\|\Omega\|_{L^{q,1}\log L(\mathbb{S}^{n-1})} & \text{if } \Omega\in L^{q,1}\log L(\mathbb{S}^{n-1}).\\
\end{cases}
\]
\end{theorem*}
The preceding sparse domination was proved to be a very useful result in \cite{LPRRR}. Among other estimates, the following $A_1-A_\infty$ estimate was established in that paper (see Lemma \ref{Lem:RHI} in Section \ref{Sec:Preliminaries} for the definition of the $A_\infty$ constant)
\[\|T_\Omega\|_{L^p(w)}\leq c_n\|\Omega\|_{L^\infty(\mathbb{S}^{n-1})}[w]_{A_1}^\frac{1}{p}[w]_{A_\infty}^\frac{1}{p'}.\]
The preceding inequality is an improvement of the following estimate established earlier in \cite{PRRR}
\[\|T_\Omega\|_{L^p(w)}\leq c_n\|\Omega\|_{L^\infty(\mathbb{S}^{n-1})}[w]_{A_1}^\frac{1}{p}[w]_{A_\infty}^{1+\frac{1}{p'}}.\]

Now we recall that the commutator of an operator $T$ and a symbol $b$ is defined as
\[[b,T]f(x)=T(bf)(x)-b(x)Tf(x).\]
In the case of $T$ being a Calder\'on-Zygmund operator this operator was introduced by R.R. Coifman, R. Rochberg and G. Weiss in \cite{CRW}. They established that $b\in \BMO$ is a sufficient condition for  $[b,T]$ to be bounded on $L^p$ for every $1<p<\infty$ and also a converse result in terms of the Riesz transforms, namely that the boundedness of $[b,R_j]$ on $L^p$ for some $1<p<\infty$ and for every Riesz transform implies that $b\in\BMO$.

In \cite{PRRR} the following estimate for commutators of rough singular integrals and a symbol $b\in\BMO$ was obtained.
\begin{equation} \label{eq:A1LORR}\|[b,T_\Omega]\|_{L^p(w)}\leq c_n\|\Omega\|_{L^\infty(\mathbb{S}^{n-1})}[w]_{A_1}^\frac{1}{p}[w]_{A_\infty}^{2+\frac{1}{p}}.
\end{equation}
One of the main goals of this paper is to improve the dependence on the $[w]_{A_\infty}$ constant in \eqref{eq:A1LORR}. Our result is the following.

\begin{theorem}\label{Thm:A1}Let $T_{\Omega}$ be a rough homogeneous singular integral with $\Omega\in L^{\infty}(\mathbb{S}^{n-1})$ and let $b\in\BMO$. For every weight $w$ we have that
\begin{equation}\label{eq:Mr}
\|[b,T_\Omega]\|_{L^p(M_r(w))\rightarrow L^p(w)}\leq c_n\|\Omega\|_{L^{\infty}(\mathbb{S}^{n-1})}\|b\|_{\BMO}(p')^3p^{2}(r')^{1+\frac{1}{p'}}
\end{equation}
where $r>1$. Assuming additionally that $w\in{A_\infty}$
\[\|[b,T_\Omega]\|_{L^p(M(w))\rightarrow L^p(w)}\leq c_n\|\Omega\|_{L^{\infty}(\mathbb{S}^{n-1})}\|b\|_{\BMO}(p')^3p^{2}[w]_{A_\infty}^{1+\frac{1}{p'}}\]
and, furthermore, if $w\in A_1$, then 
\[\|[b,T_\Omega] \|_{L^p(w)}\leq c_n\|\Omega\|_{L^{\infty}(\mathbb{S}^{n-1})}\|b\|_{\BMO}(p')^3p^{2}[w]_{A_1}^{\frac{1}{p}}[w]_{A_\infty}^{1+\frac{1}{p'}}.\]
\end{theorem}

Very recently a conjecture left open by K. Moen and A. Lerner in \cite{LM} was solved by K. Li in \cite{Li}. Actually he obtained a more general result.
\begin{theorem*}Let $T$ be a Calder\'on-Zygmund operator or a rough singular integral with $\Omega\in L^\infty(\mathbb{S}^{n-1})$. Then for every $1<q<p<\infty$
\[\|T\|_{L^p(w)}\leq c_{n,p,q} c_T [w]_{A_q^{\frac 1{p}}(A_\infty^{\exp})^{\frac 1{p'}}}\]
where
\[[w]_{A_q^{\frac 1{p}}(A_\infty^{\exp})^{\frac 1{p'}}}=\sup_Q\langle w\rangle_Q\langle w^{\frac{1}{1-q}}\rangle_Q^\frac{q-1}{p}\exp\left(\langle\log w^{-1}\rangle_Q\right)^{\frac 1{p'}}\]
and \[c_T=\begin{cases}\|\Omega\|_{L^\infty(\mathbb{S}^{n-1})} & \text{if } T=T_\Omega \text{ with } \Omega\in L^\infty(\mathbb{S}^{n-1}),\\
c_K+\|T\|_{L^2}+\|\omega\|_{\text{Dini}}& \text{if } T \text{ is an } \omega\text{-Calder\'on-Zygmund operator}. \end{cases}
\]
\end{theorem*}
This result can be regarded as an improvement of the linear dependence on the $A_q$ constant established in \cite{LPRRR}, and that, as it was stated there, follows from the linear dependence on the $A_1$ constant by \cite[Corollary 4.3]{Duo-JFA}. Such an improvement stems from the fact that 
\begin{equation}\nonumber
[w]_{A_q^{\frac 1{p}}(A_\infty^{\exp})^{\frac 1{p'}}}\leq c_n [w]_{A_q}.
\end{equation}
In the next Theorem we provide a counterpart of the preceding result for commutators.
\begin{theorem}\label{Thm:LM}Let $T$ be a Calder\'on-Zygmund operator or a rough singular integral with $\Omega\in L^\infty(\mathbb{S}^{n-1})$. Then for every $1<q<p<\infty$
\begin{equation}\label{eq:MixAq}\|[b,T]\|_{L^p(w)}\leq c_{n,p,q} c_T [w]_{A_\infty}[w]_{A_q^{\frac 1{p}}(A_\infty^{\exp})^{\frac 1{p'}}}
\end{equation}
\end{theorem}
We would like to recall the following known estimates. 
\[\begin{split}
\|[b,T]\|_{L^p(w)} &\leq c [w]_{A_q}^2,\\
\|[b,T_\Omega]\|_{L^p(w)} &\leq c [w]_{A_q}^3.
\end{split}\]
The first of them can be derived as a consequence of the quadratic dependence on the $A_1$ constant of $[b,T]$ obtained in \cite{OC} combined with \cite[Corollary 4.3]{Duo-JFA}, while the second one was established in \cite{PRRR}.
In both cases we improve the dependence on the $A_q$ constant since we are able to prove a mixed $A_\infty-A_q^{\frac 1{p}}(A_\infty^{\exp})^{\frac 1{p'}}$ bound and \[\max\{[w]_{A_\infty}, [w]_{A_q^{\frac 1{p}}(A_\infty^{\exp})^{\frac 1{p'}}}\}\leq c_n [w]_{A_q}.\]

In order to establish Theorems \ref{Thm:LM} and \ref{Thm:A1} we will rely upon a sui\-ta\-ble sparse do\-mi\-na\-tion result for $[b,T_\Omega]$. This result will be a natural bilinear counterpart of the result obtained in \cite{LORR} for $[b,T]$ with $T$ a Calder\'on-Zygmund operator and also of \eqref{eq:SpTO}. The precise statement is the fo\-llowing.

\begin{theorem}\label{Thm:Sparse} Let $T_{\Omega}$ be a rough homogeneous singular integral with $\Omega\in L^{\infty}(\mathbb{S}^{n-1})$.
Then, for every compactly supported $f,g\in \mathcal{C}^\infty({\mathbb R}^n)$ every $b\in \BMO$ and $1<p<\infty$, there exist $3^n$ dyadic lattices $\mathcal{D}_j$ and $3^n$ sparse families ${\mathcal S}_j\subset\mathcal{D}_j$ such that
\begin{equation}
|\langle [b,T_\Omega]f,g \rangle|\le C_np'\|\Omega\|_{L^{\infty}(\mathbb{S}^{n-1})}\sum_{j=1}^\infty\left(\mathcal{T}_{\mathcal{S}_j,1,p}(b,f,g)+\mathcal{T}^*_{\mathcal{S}_j,1,p}(b,f,g)\right) 
\end{equation}
where
\[
\begin{split}
\mathcal{T}_{\mathcal{S}_j,r,s}(b,f,g)&=\sum_{Q\in {\mathcal S}_j}\langle f\rangle_{r,Q}\langle (b-b_Q)g\rangle_{s,Q}|Q|\\
\mathcal{T}^*_{\mathcal{S}_j,r,s}(b,f,g)&=\sum_{Q\in {\mathcal S}_j}\langle(b-b_Q)f\rangle_{r,Q}\langle g\rangle_{s,Q}|Q|
\end{split}
\]
\end{theorem}
\begin{remark}
In the preceding Theorem and throughout the rest of this work $\langle h \rangle^w_{\alpha,Q}=\left(\frac{1}{w(Q)}\int_{Q}|h|^\alpha wdx\right)^\frac{1}{\alpha}$. We may drop $\alpha$ in the case $\alpha=1$ and $w$ when we consider the Lebesgue measure.
\end{remark}
The rest of the paper is organized as follows. We devote Section \ref{Sec:Preliminaries} to gather some results and definitions that will be needed to prove the main theorems. Section \ref{Sec:Sparse} is devoted to the proof of Theorem \ref{Thm:Sparse}. In Section \ref{sec:PThmA1} we prove Theorem \ref{Thm:A1}. We end this work providing a proof of Theorem \ref{Thm:LM} in Section \ref{Sec:ThmLM}.

\section{Preliminaries}\label{Sec:Preliminaries}
In this section we gather some definitions and results that will be necessary for the proofs of the main theorems. 

We start borrowing some definitions and a basic lemma from \cite{LN}.
Given a cube $Q_0\subset {\mathbb R}^n$, we denote by ${\mathcal D}(Q_0)$ the family of all dyadic cubes with respect to $Q_0$, namely, the cubes obtained subdividing repeatedly $Q_0$ and each of its descendants into $2^n$ subcubes of the same sidelength.

We say that  $\mathcal{D}$ is a dyadic lattice if it is a collection of cubes of  ${\mathbb R}^n$ such that:
\begin{enumerate}
\item If $Q\in\mathcal{D}$, then $\mathcal{D}(Q_0)\subset\mathcal{D}$.
\item For every pair of cubes $Q',Q''\in \mathcal{D}$ there exists a common ancestor, namely, we can find $Q\in\mathcal{D}$ such that $Q',Q''\in {\mathcal D}(Q)$.
\item For every compact set $K\subset {\mathbb R}^n$, there exists a cube $Q\in \mathcal{D}$ such that $K\subset Q$.
\end{enumerate}

\begin{lemma}[$3^n$ dyadic lattices lemma]\label{Lem:3nDL}
Given a dyadic lattice $\mathcal{D}$, there exist $3^n$ dyadic lattices $\mathcal{D}_1, \dots,\mathcal{D}_{3^n}$ such that
\[\{3Q: Q\in\mathcal{D}\}=\cup_{j=1}^{3^n}\mathcal{D}_{j}\]
and for each cube $Q\in \mathcal{D}$ and $j=1,\dots,3^n$, there exists a unique cube $R\in \mathcal{D}_j$ with
sidelength $l(R)=3l(Q)$ containing $Q$.
\end{lemma}

Now we gather some results that will be needed to prove  Theorem \ref{Thm:A1}. The first of them is the so called Reverse H\"older inequality that was proved in \cite{HP} (see also \cite{HPR}).
\begin{lemma}\label{Lem:RHI}
For every $w\in A_\infty$, namely for every weight such that 
\[[w]_{A_\infty}=\sup_Q\frac{1}{w(Q)}\int_Q M(w\chi_Q)<\infty,\]
the following estimate holds
\[\left(\frac{1}{|Q|}\int_Qw^{r_w}\right)^\frac{1}{r_w}\leq 2 \left(\frac{1}{|Q|}\int_Qw\right)\]
where $r_w=1+\frac{1}{\tau_n[w]_{A_\infty}}$ and $\tau_n>0$ is a constant independent of $w$. 
\end{lemma}
At this point we would like to recall that if $w\in A_p\subseteq A_\infty$ then $[w]_{A_\infty}\leq c_n[w]_{A_p}$. This fact makes mixed $A_\infty-A_p$ bounds interesting, since they provide a sharper dependence than $A_p$ bounds.
We also need  to borrow the following lemma from \cite{PRRR}.
\begin{lemma}\label{4.1HP2Gen}
Let $w\in A_\infty$. Let $\mathcal{D}$ be a dyadic lattice and $\mathcal{S}\subset\mathcal{D}$ be an $\eta$-sparse family. Let $\Psi$ be a Young function. Given a measurable function $f$  on $\mathbb{R}^n$ define
\[
\mathcal{B}_\mathcal{S}f(x):=\sum_{Q\in\mathcal{S}}\|f\|_{\Psi(L),Q}\chi_Q(x).
\] 
Then we have 
\[
\|\mathcal{B}_\mathcal{S}f\|_{L^1(w)}\leq \frac{4}{\eta}[w]_{A_\infty}\|M_{\Psi(L)}f\|_{L^1(w)}.
\]
\end{lemma}
We recall that $\Psi:[0,\infty)\rightarrow[0,\infty)$ is a Young function if it is a convex, increasing function such that $\Psi(0)=0$. We define the local Orlicz norm associated to a Young function $\Psi$ as 
\[\|f\|_{\Psi(L)(\mu),E}=\inf\left\{ \lambda>0:\frac{1}{\mu(E)} \int_E\Psi \left(\frac{|f|}{\lambda}\right) d\mu\leq 1\right\}\]
where $E$ is a set of finite measure. We note that in the case $\Psi(t)=t^r$ we recover the standard $L^r$ local norm. We shall drop $\mu$ from the notation in the case of the Lebesgue measure and write $w$ instead of $wdx$ for measures that are absolutely continuous with respect to the Lebesgue measure. 

Using the local the preceding definition of local norm, we can define the maximal function associated to a Young function $\Psi$  in the natural way,  
\[M_{\Psi(L)} f(x)=\sup_{x\in Q}\|f\|_{\Psi(L)(\mu),Q}.\]
We end this section recalling two basic estimates that work for doubling measures. The first of them is a particular case of the generalized H\"older inequality and the second can be derived, for example, from \cite[Lemma 4.1]{CLO}.

\begin{align}
\begin{split}
&\frac{1}{\mu(Q)}\int_Q|f-f_Q||g|d\mu\leq\|f-f_Q\|_{\exp L(\mu),Q}\|g\|_{L\log L(\mu),Q} \\
&\leq c_n\|b\|_{\BMO(\mu)}\|g\|_{\exp L(\mu),Q}\\
\end{split}\label{HGenBMOLlogL}\\
&\|f\|_{\exp L(\mu),Q}\leq c_nr'\left(\frac{1}{\mu(Q)}\int_Qw^r d\mu\right)^\frac{1}{r}\quad r>1 \label{eq:desig}
\end{align}
For a detailed account of local Orlicz norms and maximal functions associated to Young functions we encourage the reader to consult references such as \cite{RR}, \cite{PKJF}, \cite{PLondon} or \cite{CUMP}.

\section{Proof of Theorem \ref{Thm:Sparse}}\label{Sec:Sparse}
The proof of Theorem \ref{Thm:Sparse} relies upon techniques recently developed by A. K. Lerner in \cite{LRough}. Given an operator $T$ we define the bilinear operator $\mathcal{M}_T$ by
\[\mathcal{M}_T(f,g)(x)=\sup_{Q\ni x}\frac{1}{|Q|}\int_Q|T(f\chi_{{\mathbb R}^n\setminus 3Q})||g|dy,\]
where the supremum is taken over all cubes $Q\subset {\mathbb R}^n$ containing $x$.
Our first result provides a sparse domination principle based on that bilinear operator. 
\begin{theorem}\label{Thm:Sparse1}
Let $1\le q\le r$ and $s\ge 1$. Assume that $T$ is a sublinear operator of weak type $(q,q)$, and ${\mathcal M}_T$ maps
$L^r\times L^s$ into $L^{\nu,\infty}$, where $\frac{1}{\nu}=\frac{1}{r}+\frac{1}{s}$.
Then, for every compactly supported $f,g\in \mathcal{C}^\infty({\mathbb R}^n)$ and every $b\in \BMO$, there exist $3^n$ dyadic lattices $\mathcal{D}_j$ and $3^n$ sparse families ${\mathcal S}_j\subset\mathcal{D}_j$ such that
\begin{equation}\label{mainin}
|\langle [b,T]f,g \rangle|\le K\sum_{j=1}^\infty\left(\mathcal{T}_{\mathcal{S}_j,r,s}(b,f,g)+\mathcal{T}^*_{\mathcal{S}_j,r,s}(b,f,g)\right)
\end{equation}
where
\[
\begin{split}
\mathcal{T}_{\mathcal{S}_j,r,s}(b,f,g)&=\sum_{Q\in {\mathcal S}_j}\langle f\rangle_{r,Q}\langle (b-b_Q)g\rangle_{s,Q}|Q|\\
\mathcal{T}^*_{\mathcal{S}_j,r,s}(b,f,g)&=\sum_{Q\in {\mathcal S}_j}\langle(b-b_Q)f\rangle_{r,Q}\langle g\rangle_{s,Q}|Q|
\end{split}
\]
and $$K=C_n\big(\|T\|_{L^q\to L^{q,\infty}}+\|{\mathcal M}_T\|_{L^r\times L^s\to L^{\nu,\infty}}\big).$$
\end{theorem}
It is possible to relax the condition imposed on $b$ for this result and the subsequent ones, but we restrict ourselves to this choice for the sake of clarity.
\begin{proof}[Proof of Theorem \ref{Thm:Sparse1}]By Lemma \ref{Lem:3nDL}, there exist $3^n$ dyadic lattices ${\mathcal D}_j$ such that
for every $Q\subset {\mathbb R}^n$, there is a cube $R=R_Q\in {\mathcal D}_j$ for some $j$, for which
$3Q\subset R_Q$ and $|R_Q|\le 9^n|Q|$.

Let us fix a cube $Q_0\subset {\mathbb R}^n$. Now we can define a local analogue of ${\mathcal M}_T$ by
$${\mathcal M}_{T,Q_0}(f,g)(x)=\sup_{Q\ni x,Q\subset Q_0}\frac{1}{|Q|}\int_Q|T(f\chi_{3Q_0\setminus 3Q})||g|dy.$$
We define the sets $E_i$ $i=1,\dots 4$ as follows
\[\begin{split}E_1&=\{x\in Q_0: |T(f\chi_{3Q_0})(x)|>A_1\langle f\rangle_{q,3Q_0}\},\\
E_2&=\{x\in Q_0: {\mathcal M}_{T,Q_0}(f,g(b-b_{R_{Q_0}}))(x)>A_2\langle f\rangle_{r,3Q_0}\langle g(b-b_{R_{Q_0}})\rangle_{s,Q_0}\},\\
E_3&=\{x\in Q_0: |T(f\chi_{3Q_0}(b-b_{R_{Q_0}}))(x)|>A_3\langle f(b-b_{R_{Q_0}})\rangle_{q,3Q_0}\},\\
E_4&=\{x\in Q_0: {\mathcal M}_{T,Q_0}(f(b-b_{R_{Q_0}}),g)(x)>A_4\langle (b-b_{R_{Q_0}})f\rangle_{r,3Q_0}\langle g\rangle_{s,Q_0}\}.
\end{split}
\]
We can choose $A_i$ in such a way that
\[\max(|E_1|,|E_2|,|E_3|,|E_4|)\le \frac{1}{2^{n+5}}|Q_0|.\]
Actually it suffices to take 
\[A_1,A_3=(c_n)^{1/q}\|T\|_{L^q\to L^{q,\infty}}\quad\text{and}\quad A_2,A_4=c_{n,r,\nu}\|{\mathcal M}_T\|_{L^r\times L^s\to L^{\nu,\infty}}\]
with $c_n, c_{n,r,\nu}$ large enough.
For this choice of $E_i$ the set $\Omega=\cup_iE_i$ satisfies $|\Omega|\le\frac{1}{2^{n+2}}|Q_0|$.

Now applying Calder\'on-Zygmund decomposition to the function $\chi_{\Omega}$ on $Q_0$ at height $\lambda=\frac{1}{2^{n+1}}$
we obtain pairwise disjoint cubes $P_j\in {\mathcal D}(Q_0)$ such that
$$\frac{1}{2^{n+1}}|P_j|\le |P_j\cap E|\le \frac{1}{2}|P_j|$$
and also $|\Omega\setminus \cup_jP_j|=0$. From the properties of the cubes it readily follows that $\sum_j|P_j|\le \frac{1}{2}|Q_0|$ and $P_j\cap \Omega^{c}\not=\emptyset$.

Now, since $|\Omega\setminus \cup_jP_j|=0$, we have that
\[\begin{split}
\int_{Q_0\setminus \cup_jP_j}|T(f\chi_{3Q_0})||(b-b_{R_{Q_0}})g|&\le A_1\langle f\rangle_{q,3Q_0}\int_{Q_0}|g(b-b_{R_{Q_0}})|\\
\int_{Q_0\setminus \cup_jP_j}|T((b-b_{R_{Q_0}})f\chi_{3Q_0})||g|&\le A_3\langle (b-b_{R_{Q_0}})f\rangle_{q,3Q_0}\int_{Q_0}|g|.
\end{split}\]
Also, since $P_j\cap \Omega^{c}\not=\emptyset$, we obtain
\[\begin{split}\int_{P_j}|T((b-b_{R_{Q_0}})f\chi_{3Q_0\setminus 3P_j})||g|\le A_2\langle (b-b_{R_{Q_0}})f\rangle_{r,3Q_0}\langle g\rangle_{s,Q_0}|Q_0|\\
\int_{P_j}|T(f\chi_{3Q_0\setminus 3P_j})||(b-b_{R_{Q_0}})g|\le A_4\langle f\rangle_{r,3Q_0}\langle (b-b_{R_{Q_0}})g\rangle_{s,Q_0}|Q_0|.
\end{split}\]

Our next step is to observe that for any arbitrary pairwise disjoint cubes $P_j\in {\mathcal D}(Q_0)$,
\[
\begin{split}
&\int_{Q_0}|[b,T](f\chi_{3Q_0})||g|\\
&=\int_{Q_0\setminus \cup_jP_j}|[b,T](f\chi_{3Q_0})||g|+\sum_j\int_{P_j}|[b,T](f\chi_{3Q_0})||g|\\
&\le \int_{Q_0\setminus \cup_jP_j}|[b,T](f\chi_{3Q_0})||g|+\sum_j\int_{P_j}|[b,T](f\chi_{3Q_0\setminus3P_j})||g|\\
&+\sum_j\int_{P_j}|[b,T](f\chi_{3P_j})||g|.
\end{split}
\]
For the first two terms, using that $[b,T]f=[b-c,T]f$ for any $c\in {\mathbb R}$, we obtain
\begin{eqnarray}
&& \int_{Q_0\setminus \cup_jP_j}|[b,T](f\chi_{3Q_0})||g|+\sum_j\int_{P_j}|[b,T](f\chi_{3Q_0\setminus3P_j})||g| \nonumber\\
&&\le \int_{Q_0\setminus\cup_jP_j}|b-b_{R_{Q_0}}||T(f\chi_{3Q_0})||g|+\sum_j\int_{P_j}|b-b_{R_{Q_0}}||T(f\chi_{3Q_0\setminus 3P_j})||g|\nonumber\\
&&+\int_{Q_0\setminus\cup_jP_j}|T((b-b_{R_{Q_0}})f\chi_{3Q_0})||g|+\sum_j\int_{P_j}|T((b-b_{R_{Q_0}})f\chi_{3Q_0\setminus 3P_j})||g|.\nonumber
\end{eqnarray}
Therefore, combining all the preceding estimates with H\"older's inequality (here we take into account $q\le r$ and $s\ge 1$) and calling $A=\sum_i A_i$ we have that
\[\begin{split}
&\int_{Q_0}|[b,T](f\chi_{3Q_0})||g|\le\sum_j\int_{P_j}|[b,T](f\chi_{3P_j})||g|\\
 &+A \left(\langle f\rangle_{r,3Q_0}\langle (b-b_{R_{Q_0}})g\rangle_{s,Q_0}|Q_0|+\langle(b-b_{R_{Q_0}}) f\rangle_{r,3Q_0}\langle g\rangle_{s,Q_0}|Q_0|\right).
\end{split}
\]

Since $\sum_j|P_j|\le \frac{1}{2}|Q_0|$, iterating the above estimate, we obtain that there is a $\frac{1}{2}$-sparse family ${\mathcal F}\subset {\mathcal D}(Q_0)$
such that
\begin{equation}
\label{itpart}
\begin{split}
\int_{Q_0}|[b,T](f\chi_{3Q_0})||g|&\le A\sum_{Q\in {\mathcal F}}\langle (b-b_{R_{Q}})f\rangle_{r,3Q}\langle g\rangle_{s,Q}|Q|\\
&+ A\sum_{Q\in {\mathcal F}}\langle f\rangle_{r,3Q}\langle g(b-b_{R_{Q}})\rangle_{s,Q}|Q|
\end{split}
\end{equation}
To end the proof, take now a partition of ${\mathbb R}^n$ by cubes $R_j$ such that $\text{supp}\,(f)\subset 3R_j$ for each $j$. One way to do that is the following. We take a cube $Q_0$ such that
$\text{supp}\,(f)\subset Q_0$ and cover $3Q_0\setminus Q_0$ by $3^n-1$ congruent cubes $R_j$. Each of them satisfies $Q_0\subset 3R_j$. We continue covering in the same way $9Q_0\setminus 3Q_0$, and so on. The family of the resulting cubes of this process, including $Q_0$, satisfies the desired property.

Having such a partition, apply (\ref{itpart}) to each $R_j$. We obtain a $\frac{1}{2}$-sparse family ${\mathcal F}_j\subset {\mathcal D}(R_j)$ such that
\[
\begin{split}
\int_{R_j}|[b,T](f)||g|&\le A\sum_{Q\in {\mathcal F}_j}\langle (b-b_{R_{Q}})f\rangle_{r,3Q}\langle g\rangle_{s,Q}|Q|\\
&+ A\sum_{Q\in {\mathcal F}_j}\langle f\rangle_{r,3Q}\langle g(b-b_{R_{Q}})\rangle_{s,Q}|Q|
\end{split}
\]
Therefore, setting $\mathcal{F}=\cup_j\mathcal{F}_j$
\[\begin{split}
\int_{\mathbb{R}^n}|[b,T](f)||g|&\le A\sum_{Q\in {\mathcal F}}\langle (b-b_{R_{Q}})f\rangle_{r,3Q}\langle g\rangle_{s,Q}|Q|\\
&+ A\sum_{Q\in {\mathcal F}}\langle f\rangle_{r,3Q}\langle g(b-b_{R_{Q}})\rangle_{s,Q}|Q|.
\end{split}\]

Now since $3Q\subset R_Q$ and $|R_Q|\le 3^n|3Q|$, clearly $\langle h\rangle_{\alpha,{3Q}}\le c_n\langle h \rangle_{\alpha,{R_Q}}$. Further,
setting ${\mathcal S}_j=\{R_Q\in {\mathcal D}_j:Q\in {\mathcal F}\}$, and using that ${\mathcal F}$ is $\frac{1}{2}$-sparse,
we obtain that each family ${\mathcal S}_j$ is $\frac{1}{2\cdot 9^n}$-sparse. Hence

\[\begin{split}
\int_{\mathbb{R}^n}|[b,T](f)||g|&\le c_nA\sum_{j=1}^{3^n}\sum_{R\in {\mathcal S}_j}\langle (b-b_{R})f\rangle_{r,R}\langle g\rangle_{s,R}|R|\\
&+  c_nA\sum_{j=1}^{3^n}\sum_{R\in {\mathcal S}_j}\langle f\rangle_{r,R}\langle g(b-b_{R})\rangle_{s,R}|R|
\end{split}\]
and (\ref{mainin}) holds.
\end{proof}

Given $1\le p\le \infty$, we define the maximal operator ${\mathcal M}_{p,T}$ by
$${\mathcal M}_{p,T}f(x)=\sup_{Q\ni x}\left(\frac{1}{|Q|}\int_Q|T(f\chi_{{\mathbb R}^n\setminus 3Q})|^pdy\right)^{1/p}$$
(in the case $p=\infty$ we call ${\mathcal M}_{p,T}f(x)=M_Tf(x)$).

Our next step is to provide a suitable version of \cite[Corollary 3.2]{LRough} for the commutator. The result is the following.
\begin{cor}\label{Cor:Int}
Let $1\le q\le r$ and $s\ge 1$. Assume that $T$ is a sublinear operator of weak type $(q,q)$, and ${\mathcal M}_{s',T}$ is of weak type $(r,r)$.
Then, for every compactly supported $f,g\in \mathcal{C}^\infty({\mathbb R}^n)$ and every $b\in \BMO$, there exist $3^n$ dyadic lattices $\mathcal{D}_j$ and $3^n$ sparse families ${\mathcal S}_j\subset\mathcal{D}_j$ such that
\begin{equation}
|\langle [b,T]f,g \rangle|\le K\sum_{j=1}^\infty\left(\mathcal{T}_{\mathcal{S}_j,r,s}(b,f,g)+\mathcal{T}^*_{\mathcal{S}_j,r,s}(b,f,g)\right)
\end{equation}
where
\[
\begin{split}
\mathcal{T}_{\mathcal{S}_j,r,s}(b,f,g)&=\sum_{Q\in {\mathcal S}_j}\langle f\rangle_{r,Q}\langle (b-b_Q)g\rangle_{s,Q}|Q|\\
\mathcal{T}^*_{\mathcal{S}_j,r,s}(b,f,g)&=\sum_{Q\in {\mathcal S}_j}\langle(b-b_Q)f\rangle_{r,Q}\langle g\rangle_{s,Q}|Q|
\end{split}
\]
and
\[K=C_n\left(\|T\|_{L^q\to L^{q,\infty}}+\|{\mathcal M}_{s',T}\|_{L^r\to L^{r,\infty}}\right).\]
\end{cor}

\begin{proof} The proof is the same as \cite[Corollary 3.2]{LRough}. 
It suffices to observe that
\[
\|{\mathcal M}_T\|_{L^r\times L^s\to L^{\nu,\infty}}\le C_n\|{\mathcal M}_{s',T}\|_{L^r\to L^{r,\infty}}\quad(1/\nu=1/r+1/s),
\]
and to apply Theorem \ref{Thm:Sparse1}.
\end{proof}
\begin{remark}\label{rem:Sparse}
At this point we would like to note that if $T$ is an $\omega$-Calder\'on-Zygmund operator, with $\omega$ satisfying a Dini condition, since $M_T$ is of weak-type $(1,1)$ with  \[\|M_T\|_{L^1\rightarrow L^{1,\infty}}\leq c_n\left(C_K+\|T\|_{L^2}+\|\omega\|_{\text{Dini}}\right)\] (see \cite{L}, also for the notation) and we have that \[\|T\|_{L^1\rightarrow L^{1,\infty}}\leq c_n\left(\|T\|_{L^2}+\|\omega\|_{\text{Dini}}\right),\]then from the preceding Corollary we recover a bilinear version of the sparse domination established in \cite{LORR}.
\end{remark}

In order to use Corollary \ref{Cor:Int} to obtain Theorem \ref{Thm:Sparse}, we need to borrow some results from \cite{LRough}.
Given an operator $T$, we define the maximal operator $M_{\lambda, T}$ by
\[M_{\lambda, T}f(x)=\sup_{Q\ni x}(T(f\chi_{{\mathbb R}^n\setminus 3Q})\chi_Q)^*(\lambda|Q|)\quad 0<\lambda<1.\]
That operator was proved to be of weak type $(1,1)$ in \cite{LRough} where the following estimate was established.
\begin{theorem}\label{Thm:logest} If $\Omega\in L^{\infty}(\mathbb{S}^{n-1})$, then
\begin{equation}\label{logestim}
\|M_{\lambda, T_{\Omega}}\|_{L^1\to L^{1,\infty}}\le C_n\|\Omega\|_{L^{\infty}(\mathbb{S}^{n-1})}\Big(1+\log\frac{1}{\lambda}\Big)\quad 0<\lambda<1.
\end{equation}
\end{theorem}
Also in  \cite{LRough} the following result showing the relationship between the $L^1\to L^{1,\infty}$ norms of the operators $M_{\lambda, T}$ and ${\mathcal M}_{p,T}$ was provided.

\begin{lemma}\label{Lem:equiv} Let $0<\gamma\le 1$ and let $T$ be a sublinear operator. The following statements are equivalent:
\begin{enumerate}
\item
there exists $C>0$ such that for all $p\ge 1$,
\[\|{\mathcal M}_{p,T}f\|_{L^1\to L^{1,\infty}}\le Cp^{\gamma};\]
\item
there exists $C>0$ such that for all $0<\lambda<1$,
\[\|M_{\lambda,T}f\|_{L^1\to L^{1,\infty}}\le C\left(1+\log\frac{1}{\lambda}\right)^{\gamma}.\]
\end{enumerate}
\end{lemma}

At this point we are in the position to prove that Theorem \ref{Thm:Sparse} follows as a corollary from the previous results.
\begin{proof}[Proof of Theorem \ref{Thm:Sparse}] Theorem \ref{Thm:logest} combined with Lemma \ref{Lem:equiv} with $\gamma=1$ yields
\[\|{\mathcal M}_{p,T_{\Omega}}\|_{L^1\to L^{1,\infty}}\le c_np\|\Omega\|_{L^{\infty}(\mathbb{S}^{n-1})}\] with $p\geq 1$. Also, by \cite{S}, we have that \[\|T_{\Omega}\|_{L^1\to L^{1,\infty}}\le C_n\|\Omega\|_{L^{\infty}(\mathbb{S}^{n-1})}.\]
Hence, by Corollary \ref{Cor:Int} with $q=r=1$ and $s=p>1$,  there exist $3^n$ dyadic lattices $\mathcal{D}_j$ and $3^n$ sparse families ${\mathcal S}_j\subset{\mathcal D}_j$ such that
\[|\langle [b,T_{\Omega}]f,g \rangle|\le C_np'\|\Omega\|_{L^{\infty}(\mathbb{S}^{n-1})}\sum_{j=1}^{3^n}\left(\mathcal{T}_{\mathcal{S}_j,1,p}(b,f,g)+\mathcal{T}_{\mathcal{S}_j,1,p}(b,f,g)\right).\]\end{proof}
\section{Proof of Theorem \ref{Thm:A1}}\label{sec:PThmA1}
We start providing a proof for \eqref{eq:Mr}. We follow some of the key ideas from \cite{LOP1,LOP2} (see also \cite{PRRR}).
By duality, it suffices to prove \eqref{eq:Mr} it suffices to show that
\[
\left\| \frac{[b,T_\Omega]f}{M_{r}w}\right\|_{L^{p'}(M_{r}w)}\leq c_n\|\Omega\|_{L^{\infty}(\mathbb{S}^{n-1})}\|b\|_{\BMO}(p')^3p^{2}(r')^{1+\frac{1}{p'}}
\Big\| \frac{f}{w}\Big\|_{L^{p'}(w)}.
\]
We can calculate the norm by duality. Then,
\[
\bigg\| \frac{[b,T_\Omega]f}{M_{r}w}\bigg\|_{L^{p'}(M_{r}w)}=\sup_{\|h\|_{L^{p}(M_{r}w)}=1}\left|
 \int_{\mathbb{R}^{n}}  [b,T_{\Omega}]f(x) h(x)dx\right|. 
\]
Let us define now a Rubio de Francia algorithm suited for this si\-tua\-tion (see \cite[Chapter IV.5]{GCRdF} and \cite{CUMP} for plenty of applications of the Rubio de Francia algorithm).
First we consider the operator
\[S(f)=\frac{M\big(f(M_{r}w)^{\frac{1}{p}}\big)}{(M_{r}w)^{\frac{1}{p}}}\]
and we observe that $S$ is bounded on $L^{p}(M_{r}w)$ with norm bounded by a dimensional multiple of $p'$.
Relying upon $S$ we define
\[
R(h)=\sum_{k=0}^{\infty}\frac{1}{2^{k}}\frac{S^{k}h}{\|S\|_{L^{p}(M_{r}w)}^{k}}.
\]
This operator has the following properties:
\begin{enumerate}
\item[(a)] $0\leq h\leq R(h)$,
\item[(b)] $\|Rh\|_{L^{p}(M_{r}w)}\leq2\|h\|_{L^{p}(M_{r}w)}$,
\item[(c)] $R(h)(M_{r}w)^{\frac{1}{p}}\in A_{1}$ with $[R(h)(M_{r}w)^{\frac{1}{p}}]_{A_{1}}\leq cp'$. We also note that $[Rh]_{A_{\infty}}\leq[Rh]_{A_{3}}\leq c_{n}p'$.
\end{enumerate}
Using Theorem \ref{Thm:Sparse} and taking into account (a) we have that,
\[\begin{split}
&\left|
 \int_{\mathbb{R}^{n}}  [b,T_{\Omega}]f(x) h(x)dx\right|\\
&\leq C_ns'\|\Omega\|_{L^{\infty}(\mathbb{S}^{n-1})}\sum_{j=1}^\infty\left(\mathcal{T}_{\mathcal{S}_j,1,s}(b,f,h)+\mathcal{T}^*_{\mathcal{S}_j,1,s}(b,f,h)\right)\\
&\leq C_ns'\|\Omega\|_{L^{\infty}(\mathbb{S}^{n-1})}\sum_{j=1}^\infty\left(\mathcal{T}_{\mathcal{S}_j,1,s}(b,f,Rh)+\mathcal{T}^*_{\mathcal{S}_j,1,s}(b,f,Rh)\right)
\end{split}\]
and it suffices to obtain estimates for 
\[
I:=\mathcal{T}_{\mathcal{S}_j,1,s}(b,f,Rh)\qquad\text{and}\qquad II:=\mathcal{T}^*_{\mathcal{S}_j,1,s}(b,f,Rh).
\]
First we focus on $I$. Now we choose $r,s>1$ such that $sr=1+\frac{1}{\tau_n[Rh]_{A_\infty}}$. For instance, choosing $r=1+\frac{1}{2\tau_n[Rh]_{A_\infty}}$ we have that $s=2\frac{1+\tau_n[Rh]_{A_\infty}}{1+2\tau_n[Rh]_{A_\infty}}$ and also that  $sr'=2(1+\tau_n[Rh]_{A_\infty})\simeq[Rh]_{A_\infty}$.
Now we recall that for every $0<t<\infty$ it was established in \cite[Corollary 3.1.8]{GMod} that 
\[\left(\frac{1}{|Q|}\int_{Q}|b(x)-b_{Q}|^t dx\right)^{\frac{1}{t}}\leq \left(t\Gamma(t)\right)^\frac{1}{t}e^{\frac{1}{t}+1}2^n\|b\|_{\BMO}\]
For $t>1$ we have that $\left(t\Gamma(t)\right)^\frac{1}{t}e^{\frac{1}{t}+1}2^n\leq c_nt$. 
Taking into account the preceding estimate, the choices for $r$ and $s$, the reverse H\"older inequality (Lemma \ref{Lem:RHI}), and the property (c) above, we have that
\begin{align*}I &\leq\sum_{Q\in\mathcal{S}_{j}}\left(\frac{1}{|Q|}\int_{Q}|b(x)-b_{Q}|^s|Rh(x)|^sdx\right)^{\frac{1}{s}}\int_{Q}|f|dy \\
&\leq\sum_{Q\in\mathcal{S}_{j}} \langle b-b_Q\rangle_{sr',Q}\langle Rh \rangle_{sr,Q}\int_{Q}|f|dy \\
 & \leq c_n(sr')\|b\|_{\BMO}\sum_{Q\in\mathcal{S}_{j}}\left(\frac{1}{|Q|}\int_{Q}Rh\right)\int_{Q}|f|dy\\
 & \leq c_n[Rh]_{A_{\infty}}\|b\|_{\BMO}\sum_{Q\in\mathcal{S}_{j}}Rh(Q)\frac{1}{|Q|}\int_{Q}|f|dy\\
 &\leq c_{n}p'\|b\|_{\BMO}\sum_{Q\in\mathcal{S}_{j}}Rh(Q)\frac{1}{|Q|}\int_{Q}|f|dy.
\end{align*}
An application of Lemma \ref{4.1HP2Gen} with $\Psi(t)=t$ yields
\[
\sum_{Q\in\mathcal{S}_{j}}Rh(Q)\frac{1}{|Q|}\int_{Q}|f|dy\leq8[Rh]_{A_{\infty}}\|Mf\|_{L^{1}(Rh)}\leq c_np'\|Mf\|_{L^{1}(Rh)}.
\]
From here
\[\begin{split}
\|Mf\|_{L^{1}(Rh)}&\leq\Big(\int_{\mathbb{R}^{n}}|Mf|^{p'}\big(M_{r}w\big)^{1-p'}\Big)^{\frac{1}{p'}}\Big(\int_{\mathbb{R}^{n}}(Rh)^{p}M_{r}w\Big)^{\frac{1}{p}}\\
&\leq2\Big\| \frac{Mf}{M_{r}w}\Big\| _{L^{p'}(M_{r}w)}.
\end{split}
\]
Now by \cite[Lemma 3.4]{LOP1} (see also \cite[Lemma 2.9]{OC})
\[
\Big\|\frac{Mf}{M_{r}w}\Big\|_{L^{p'}(M_{r}w)}\leq c p(r')^{\frac{1}{p'}}\Big\|\frac{f}{w}\Big\|_{L^{p'}(w)}.
\]
Gathering all the preceding estimates we have that
\[
I\leq c_{n}\|b\|_{\BMO}p(p')^{3}(r')^{\frac{1}{p'}}\left\Vert \frac{f}{w}\right\Vert _{L^{p'}(w)}.
\]

Now we turn our attention to $II$. Recalling that we have chosen $rs=1+\frac{1}{\tau_n[Rh]_{A_\infty}}$, taking into account the Reverse H\"older inequality and applying also \eqref{HGenBMOLlogL} we have that
\[\begin{split}
II&\leq\sum_{Q\in\mathcal{S}_{j}}\Big(\frac{1}{|Q|}\int_{Q}|b(y)-b_{Q}|f(y)dy\Big)\langle Rh\rangle_{s,Q}|Q|\\
&\leq\sum_{Q\in\mathcal{S}_{j}}\Big(\frac{1}{|Q|}\int_{Q}|b(y)-b_{Q}|f(y)dy\Big)\langle Rh\rangle_{rs,Q}|Q|\\
&\leq c_n\|b\|_{\BMO}\sum_{Q\in\mathcal{S}_{j}}\|f\|_{L\log L,Q}Rh(Q).
\end{split}
\]
Then a direct application of Lemma \ref{4.1HP2Gen} with $\Psi(t)=t\log(e+t)$ yields the following estimate
\[
\sum_{Q\in\mathcal{S}_{j}}\|f\|_{L\log L,Q}Rh(Q)\leq8[Rh]_{A_{\infty}}\|M_{L\log L}f\|_{L^{1}(Rh)}.
\]
Arguing as in the estimate of $I$,
\[
\|M_{L\log L}f\|_{L^{1}(Rh)}\leq2\Big\|\frac{M_{L\log L}f}{M_{r}w}\Big\| _{L^{p'}(M_{r}w)}.
\]
Now \cite[Proposition 3.2]{OC} gives
\[
\Big\|\frac{M_{L\log L}f}{M_{r}w}\Big\|_{L^{p'}(M_{r}w)}\leq c_{n}p^{2}(r')^{1+\frac{1}{p'}}\Big\|\frac{f}{w}\Big\|_{L^{p'}(w)}.
\]
Combining all the estimates we have that
\[
II\leq c_{n}\|b\|_{\BMO}(p')^2p^{2}(r')^{1+\frac{1}{p'}}\Big\| \frac{f}{w}\Big\|_{L^{p'}(w)}.
\]
Finally, collecting the estimates we have obtained for $I$ and $II$, we arrive at the desired bound, namely
\[
\Big\|\frac{[b,T_\Omega]f}{M_{r}w}\Big\|_{L^{p'}(M_{r}w)}\leq c_n\|\Omega\|_{L^{\infty}(\mathbb{S}^{n-1})}\|b\|_{\BMO}(p')^3p^{2}(r')^{1+\frac{1}{p'}}\Big\|\frac{f}{w}\Big\|_{L^{p'}(w)}.
\]
We end the proof observing that the $A_\infty$ and the $A_1-A_\infty$ results are a direct consequence of the estimate we have just established and of the Reverse-H\"older inequality (see \cite{LOP1, LOP2, HP} for this kind of argument).
\qed

\section{Proof of Theorem \ref{Thm:LM}}\label{Sec:ThmLM}
Let us consider first the case in which $T$ is a Calder\'on-Zygmund operator. 
Calculating the norm by duality we have that
\[\|[b,T]f\|_{L^p(w)}=\sup_{\|g\|_{L^{p'}(w)}=1}\left|\int [b,T](f) g w\right|.\]
Now taking into account Remark \ref{rem:Sparse} (or \cite{LORR}) we have that
\[\left|\int [b,T](f)gw\right|\leq c_nc_T\sum_{j=1}^{3^n}\left(\mathcal{T}_{\mathcal{S}_j,1,1}(b,f,gw)+ \mathcal{T}^*_{\mathcal{S}_j,1,1}(b,f,gw)\right)\]
so it suffices to provide estimates for 
\[\mathcal{T}_{\mathcal{S},1,1}(b,f,gw)\quad\text{and}\quad\mathcal{T}^*_{\mathcal{S},1,1}(b,f,gw).\]
First we work on $\mathcal{T}_{\mathcal{S}_j,1,1}(b,f,gw)$. Following ideas in \cite{Li} we have that
\[\langle w\rangle_Q\langle w^\frac{1}{1-q}\rangle^{q-1}=\langle w\rangle_Q\langle\sigma^{\frac{1}{p'}}\rangle^p_{\overline{A},Q}\]
where $\overline{A}(t)=t^\frac{p}{q-1}$ and $\sigma=w^{1-p'}$. Then, choosing $s<p'$ and taking into account \cite[Lemma 6]{IFRR}, \eqref{HGenBMOLlogL} and \eqref{eq:desig},
\begin{align*}
&\mathcal{T}_{\mathcal{S},1,1}(b,f,gw)= \sum_{Q\in \mathcal S}\langle f \rangle_Q \langle g(b-b_Q)w\rangle_Q|Q| \\
&\le c   \sum_{Q\in \mathcal S}\langle fw^{\frac{1}{p}}\rangle_{A,Q}\langle w^{-\frac{1}{p}}\rangle_{\overline{A},Q} \|g\|_{L\log L(w),Q}\|b-b_Q\|_{\exp L(w),Q} \\
&\le cs' \|b\|_{\BMO}[w]_{A_\infty} \sum_{Q\in \mathcal S}\langle fw^{\frac{1}{p}}\rangle_{A,Q}\langle w^{-\frac{1}{p}}\rangle_{\overline{A},Q} \langle g\rangle^w_{s,Q}\\
&\times \exp\left(\langle \log w^{-1}\rangle_Q\right)^{\frac 1{p'}}\exp\left(\langle \log w\rangle_Q\right)^{\frac 1{p'}}\\
&\le cs' \|b\|_{\BMO}[w]_{A_\infty}[w]_{A_q^{\frac 1{p}}(A_\infty^{\exp})^{\frac 1{p'}}} \Big( \sum_{Q\in \mathcal S} \langle fw^{\frac{1}{p}}\rangle_{A,Q}^p |Q| \Big)^{\frac{1}{p}}\\
&\times \Big( \sum_{Q\in \mathcal S}\left( \langle g\rangle^w_{s,Q}\right)^{p'} \exp\left(\langle \log w^{-1}\rangle_Q\right)^{\frac 1{p'}}|Q| \Big)^{\frac 1{p'}}\\
&\le c_n\gamma^{-1}\|M_A\|_{L^p}\|b\|_{\BMO}[w]_{A_\infty} [w]_{A_q^{\frac 1{p}}(A_\infty^{\exp})^{\frac 1{p'}}}\|f\|_{L^p(w)}\|g\|_{L^{p'}(w)},
\end{align*}
where in the last step we use the Carleson embedding Theorem \cite[Theorem 4.5]{HP} and the sparsity of $\mathcal{S}$.

Now we turn our attention to $\mathcal{T}^*_{\mathcal{S},1,1}(b,f,gw)$. We observe that for any $r>1$
\begin{align*}
\mathcal{T}^*_{\mathcal{S},1,1}(b,f,gw)&= \sum_{Q\in \mathcal S}\langle f(b-b_Q)\rangle_Q \langle gw\rangle_Q|Q| \\
&\le \sum_{Q\in \mathcal S}\langle f\rangle_{r,Q}\langle b-b_Q\rangle_{r',Q} \langle gw\rangle_Q|Q|\\
&\le c\|b\|_{\BMO} \sum_{Q\in \mathcal S} \langle f\rangle_{r,Q} \langle gw\rangle_Q|Q|\end{align*}
and from this point it suffices to follow the proof of \cite[Theorem 3.1]{Li} to obtain the following estimate
\[\mathcal{T}^*_{\mathcal{S},1,1}(b,f,gw)\leq c [w]_{A_q^{\frac 1{p}}(A_\infty^{\exp})^{\frac 1{p'}}}\|f\|_{L^p(w)}\|g\|_{L^{p'}(w)}.\]
Combining the estimates for $\mathcal{T}_{\mathcal{S},1,1}(b,f,gw)$ and $\mathcal{T}^*_{\mathcal{S},1,1}(b,f,gw)$ we obtain \eqref{eq:MixAq} in the case of $T$ being a Calder\'on-Zygmund operator. 

Let us consider now the remaining case. Assume that $T$ is a rough singular integral with $\Omega\in L^\infty(\mathbb{S}^{n-1})$.
Calculating the norm by duality and denoting by  $[b,T]^t$ the adjoint of $[b,T]$ we have that
\[\|[b,T]f\|_{L^p(w)}=\sup_{\|g\|_{L^{p'}(w)}=1}\left|\int [b,T](f) g w\right|=\sup_{\|g\|_{L^{p'}(w)}=1}\left|\int [b,T]^t(gw)f\right|.\]
Taking into account that $[b,T]^t$ is also a commutator we can use the sparse domination obtained in Theorem \ref{Thm:Sparse} so we have that
\[\left|\int [b,T]^t(gw)f\right|\leq c_nu'\|\Omega\|_{L^\infty(\mathbb{S}^{n-1})}\sum_{j=1}^{3^n}\left(\mathcal{T}_{\mathcal{S}_j,u,1}(b,f,gw)+ \mathcal{T}^*_{\mathcal{S}_j,u,1}(b,f,gw)\right)\]
and then the question reduces to control both
\[\mathcal{T}_{\mathcal{S}_j,u,1}(b,f,gw)\quad \text{and}\quad \mathcal{T}^*_{\mathcal{S}_j,u,1}(b,f,gw).\]
We begin observing that, arguing as before, choosing $1<s<p'$
\[\begin{split}
\mathcal{T}_{\mathcal{S}_j,u,1}(b,f,gw)&=\sum_{Q\in {\mathcal S}_j}\langle f\rangle_{u,Q}\langle (b-b_Q)gw\rangle_{1,Q}|Q|\\
&\leq  cs'[w]_{A_\infty}\|b\|_{\BMO}\sum_{Q\in {\mathcal S}_j}\langle f\rangle_{u,Q}\langle g\rangle^w_{s,Q}w(Q)=c[w]_{A_\infty}\|b\|_{\BMO}B_1.
\end{split}
\]
On the other hand we have that for $s_1>1$ to be chosen later
\[\begin{split}
\mathcal{T}^*_{\mathcal{S},u,1}(b,f,gw)&=\sum_{Q\in {\mathcal S}}\langle(b-b_Q)f\rangle_{u,Q}\langle gw\rangle_{1,Q}|Q|\\
&\leq\sum_{Q\in {\mathcal S}}\langle f\rangle_{us_1,Q}\langle b-b_Q\rangle_{us_1',Q}\langle gw\rangle_{1,Q}|Q|\\
&\leq c\|b\|_{\BMO}\sum_{Q\in {\mathcal S}}\langle f\rangle_{us_1,Q}\langle gw\rangle_{1,Q}|Q|=c\|b\|_{\BMO} B_2.
\end{split}
\]
By H\"older inequality, we have that both $B_1$ and $B_2$ are controlled by
\[\sum_{Q\in {\mathcal S}}\langle f\rangle_{us_1,Q}\langle g\rangle^w_{s,Q}w(Q).\]
We note that we can choose $us_1$ as close to $1$ as we want so let us rename $us_1=r$. Now
denoting $\overline{B}(t)=t^{\frac{p}{r(q-1)}}$ and arguing as in \cite[Theorem 3.1]{Li} we have that
\[\begin{split}
\sum_{Q\in {\mathcal S}_j}\langle f\rangle_{r,Q}\langle g\rangle^w_{s,Q}w(Q)
&\le [w]_{A_q^{\frac 1{p}}(A_\infty^{\exp})^{\frac 1{p'}}}\Big( \sum_{Q\in \mathcal S} \langle f^r w^{\frac rp}\rangle_{B,Q}^\frac pr |Q| \Big)^{\frac 1p}\\
&\times \Big( \sum_{Q\in \mathcal S} (\langle g \rangle_{s,Q}^w)^{p'} \exp(\langle \log w\rangle_Q)|Q| \Big)^{\frac 1{p'}}\\
&\le c_n \gamma^{-1}p\|M_B\|_{L^{p/r}}^{\frac 1r}[w]_{A_q^{\frac 1{p}}(A_\infty^{\exp})^{\frac 1{p'}}}\|f\|_{L^p(w)}\|g\|_{L^{p'}(w)}
\end{split}
\]
where in the last step we have used again the sparsity of $\mathcal{S}$ and the Carleson embedding theorem (\cite[Theorem 4.5]{HP}). 
Collecting all the estimates 
\[\left|\int [b,T]^t(gw)f\right|\leq c_n\|\Omega\|_{L^\infty(\mathbb{S}^{n-1})}[w]_{A_\infty}[w]_{A_q^{\frac 1{p}}(A_\infty^{\exp})^{\frac 1{p'}}}\|f\|_{L^p(w)}\|g\|_{L^{p'}(w)}.\]
This ends the proof of Theorem  \ref{Thm:LM}.\qed

\section*{acknoledgements}The author would like to thank Kangwei Li for some insightful discussions during the elaboration of this paper.

\end{document}